\documentclass[a4paper,11pt]{amsart}

\pdfoutput=1

\usepackage{amsmath,amssymb,amsthm,url}
\usepackage[utf8]{inputenc}
\usepackage[T1]{fontenc}
\usepackage{libertine}
\usepackage[libertine,cmintegrals,cmbraces]{newtxmath}
\usepackage{bbold}
\usepackage[shortlabels]{enumitem}
\usepackage{mathtools}
\usepackage{adjustbox}
\usepackage{microtype}
\usepackage{xcolor}
\usepackage{tikz}
\usepackage{tikz-cd}
\usepackage{nicefrac}
\usepackage{dsfont}
\usepackage{todonotes}
\usepackage{a4wide}
\usepackage{quiver}
\usepackage{anyfontsize}

\AtBeginDocument{%
   \def\MR#1{}
} 
\usepackage{euscript}
\usepackage[all]{xy}

\usepackage[pagebackref]{hyperref}
  
\hypersetup{%
  bookmarksnumbered=true,
  colorlinks=true,%
  linkcolor=blue,%
  citecolor=blue,%
  filecolor=blue,%
  menucolor=blue,%
  urlcolor=blue,%
  pdfnewwindow=true,%
  pdfstartview=FitBH}
\usepackage[capitalise]{cleveref}

\usepackage{xcolor}
\definecolor{seagreen}{RGB}{46,139,87}
\definecolor{maroon}{RGB}{128,0,0}
\definecolor{darkviolet}{RGB}{148,0,211}
\definecolor{twelve}{RGB}{100,100,170}
\definecolor{thirteen}{RGB}{100,150,50}
\definecolor{fourteen}{RGB}{200,0,0}
\definecolor{fifteen}{RGB}{0,200,0}
\definecolor{sixteen}{RGB}{0,0,200}
\definecolor{seventeen}{RGB}{200,0,200}
\definecolor{eighteen}{RGB}{0,200,200}

\newcommand{\bb}[1]{\mathbb{#1}}

\newcommand{\es}[1]{\EuScript{#1}}
\renewcommand{\sf}[1]{{\mathsf{#1}}}

\DeclareMathOperator{\hocolim}{\mathsf{hocolim}}
\DeclareMathOperator{\holim}{\mathsf{holim}}

\DeclareMathOperator{\hocof}{\mathsf{hocof}}

\newcommand{\sS}{\sf{sSets}}

\newcommand{\vect}[1]{\mathsf{Vect}_{#1}}

\newcommand{\Fun}{\sf{Fun}}

\DeclareMathOperator{\Map}{\mathsf{Map}}

\DeclareMathOperator{\id}{\mathsf{Id}}

  \newcommand{\adjunction}[4]{
\xymatrix{
#1:#2 \ar@<.5ex>[r] &
\ar@<.5ex>[l] #3:#4
}}

\newtheorem{thm}{Theorem}[section]

\newtheorem{lem}[thm]{Lemma}
\newtheorem{cor}[thm]{Corollary}

\newtheorem{xxthm}{Theorem}

\newtheorem{xxcor}[xxthm]{Corollary}

\theoremstyle{definition}
\newtheorem{definition}[thm]{Definition}

\newtheorem{rem}[thm]{Remark}

\begin{document}
\title{On tested Bousfield--Friedlander localizations}
\author{Niall Taggart}
\address{Queen's University Belfast}
\email{n.taggart@qub.ac.uk}

\begin{abstract}
We show that a Bousfield--Friedlander localization of a model category of functors with respect to a set of test morphisms, as introduced by Bandklayder, Bergner, Griffiths, Johnson, and Santhanam, can be characterized as a left Bousfield localization at a specific tensoring of the set of homotopy generators. This reformulation immediately strengthens their main result by establishing that such localizations always yield cellular or combinatorial model categories.
We apply this framework to functor calculus in two key ways. First, we construct a homogeneous model structure for any calculus arising from tested Bousfield–Friedlander localization. As a consequence, we construct a Quillen adjunction between homogeneous functors in Goodwillie calculus and homogeneous functors in the discrete calculus of Bauer, Johnson, and McCarthy. Second, we show that the polynomial model structure of Weiss calculus is a specific instance of this localization framework, suggesting a unifying perspective on various flavours of functor calculus.
\end{abstract}
\maketitle

\section{Introduction}
Given a model category $\es{M}$ and a (Quillen) idempotent monad $Q: \es{M} \to \es{M}$, the Bousfield--Friedlander localization of $\es{M}$ (if it exists) is a method for universally inverting those morphisms of $\es{M}$ which become equivalences after applying $Q$. These localizations have been ubiquitous throughout homotopy theory~\cite{BousfieldFriedlander, Bousfield} and are an important tool in constructing model categories for functor calculus~\cite{BiedermannChornyRondigs, BiedermannRondigs, BarnesOman}. The localization process has many nice features~\cite{Stanculescu}, but it is not clear if localizations of this form preserve cofibrant generation. This presents a hurdle in fully understanding the model structure and our ability to perform subsequent (co)localizations is hindered.

Motivated by applications to functor calculus, Bandklayder, Bergner, Griffiths, Johnson and Santhanam~\cite{BBGJS2} gave general criteria for when the Bousfield--Friedlander localization on a  model category of simplicial functors is cofibrantly generated. These criteria included the existence of a set $T(Q)$ of \emph{test morphisms} which detect fibrations in the Bousfield--Friedlander localization. The first result of this note is that these \emph{tested Bousfield--Friedlander localizations} are always left Bousfield localizations with respect to a specific set of maps induced by the test morphisms. This result allows for now standard techniques in the case of left Bousfield localizations~\cite{Hirschhorn} to be exploited within the remit of tested Bousfield--Friedlander localizations. This observation will form the crucial ingredient in joint work of the author with Bergner, Griffiths, Johnson and Santhanam on \emph{testing for calculi}, i.e., constructing calculi from sets of test morphisms and hence providing a unifying home for many variants of functor calculus, which although formally similar often appear distinct.

\begin{xxthm}[\cref{thm: BF and LBL}]\label{main thm: LBL}
Let $Q: \Fun(\es{C}, \es{D}) \to \Fun(\es{C}, \es{D})$ be a tested Quillen idempotent monad\footnote{A homotopical endofunctor $Q: \es{M} \to \es{M}$ is a \emph{Quillen idempotent monad} if $Q$ preserves homotopy pullback squares, and if there exists a natural transformation $\eta: \bb{1} \to Q$ such that both induced natural transformations $\eta_Q, Q(\eta): Q \to Q^2$ are levelwise weak equivalences.}. If $\es{D}$ is a proper cellular simplicial model category, then there exists a set $T_\es{D}(Q)$ of morphisms in $\Fun(\es{C}, \es{D})$ such that the Bousfield--Friedlander localization at $Q$ is the left Bousfield localization of the underlying model structure at the set $T_\es{D}(Q)$.
\end{xxthm}

The set $T_\es{D}(Q)$ does not agree with the set of acyclic cofibrations produced in~\cite[Theorem 8.3]{BBGJS2} for the rather obvious reason that it is not a set of generating acyclic cofibrations. Our set is formed by tensoring the test morphisms with homotopy generators of $\es{D}$, which is homotopically well-behaved. Moreover, our set avoids the use of the pushout-product. Providing an explicit set of generating acyclic cofibrations for a left Bousfield localization is a challenging endeavour, but if $\es{D}$ is stable, then one can likely produce an explicit set of generating acyclic cofibrations analogous to \emph{loc. cit.} via the theory of stable left Bousfield localization~\cite{BarnesRoitzheim}.

As a corollary, we obtain a strengthening of the main theorem of~\cite{BBGJS2} under slightly different hypotheses, see~\cref{rem: hypotheses}.

\begin{xxcor}[\cref{cor: cof gen}]
If $\es{D}$ is a proper cellular simplicial model category, then a tested Bousfield--Friedlander localization of $\Fun(\es{C}, \es{D})$ is proper and cellular.
\end{xxcor}

We have stated the above results for cellular model categories but they also hold in the case of combinatorial model categories.

Our applications are to the study of various flavours of functor calculus. Functor calculi are a family of homotopy theoretic abstractions of differential calculus. They have had wide ranging applications throughout mathematics including applications of Goodwillie calculus~\cite{GoodCalcI, GoodCalcII, GoodCalcIII} to algebraic $K$-theory~\cite{McCarthy} and applications of Weiss calculus to study problems of geometric origin~\cite{KrannichRandal-Williams, Hu}. To a functor $F$, a calculus assigns a sequence of functors approximating $F$ 
\[\begin{tikzcd}
	&& F \\
	\cdots & {P_nF} & \cdots & {P_1F} & {P_0F}
	\arrow[from=2-4, to=2-5]
	\arrow[from=2-2, to=2-3]
	\arrow[from=2-3, to=2-4]
	\arrow[bend right=30, from=1-3, to=2-2]
	\arrow[from=2-1, to=2-2]
	\arrow[bend left=30, from=1-3, to=2-4]
	\arrow[bend left=20, from=1-3, to=2-5]
\end{tikzcd}\]
in a way analogous to the Taylor series for a function. The functor $P_nF$ is the \emph{universal degree $n$} functor under $F$, where a functor is degree $n$ if the map $F \to P_nF$ is an equivalence.

Bandklayder, Bergner, Griffiths, Johnson and Santhanam~\cite{BBGJS2} applied their theory to two forms of functor calculus: Goodwillie calculus and the discrete calculus of Bauer, Johnson and McCarthy~\cite{BauerJohnsonMcCarthy}. In both cases, degree $n$ model structures are constructed as a Bousfield--Friedlander localization~\cite{BiedermannChornyRondigs,BiedermannRondigs, BBGJS2} at the universal approximation $P_n$, and a set of test morphisms is provided.  We show that when a calculus exists via tested Bousfield--Friedlander localizations, we can also produce a model structure that captures the homogeneous functors: those functors which are degree $n$ and have trivial degree $(n-1)$ approximation. This identification crucially relies on~\cref{main thm 2} since the work of~\cite{BBGJS2} does not produce a model structure which you can \emph{a piori} right Bousfield localize. 

\begin{xxthm}[{\cref{thm: homog as RBL}}]\label{main thm 2}
Under modest hypotheses, if both a degree $n$ and degree $(n-1)$ model structure exist on $\Fun(\es{C}, \es{D})$ as tested Bousfield--Friedlander localizations, then there is a model structure on $\Fun(\es{C}, \es{D})$ in which the bifibrant objects are the bifibrant $n$-homogeneous functors.
\end{xxthm}

In~\cref{cor: homog discerete Goodwillie} we leverage~\cref{main thm 2} to show that there is a Quillen adjunction between homogeneous functors in discrete calculus and homogeneous functors in Goodwillie calculus, which is, in general, not a Quillen equivalence, i.e., these calculi are homotopically distinct, see~\cref{rem:different}.

In Weiss calculus~\cite{WeissOrthog, WeissErratum}, Barnes and Oman~\cite{BarnesOman} constructed an $n$-polynomial model structure as both a left Bousfield localization (motivated by having the fibrant objects the $n$-polynomial functors) and a Bousfield--Friedlander localization (motivated by having the universal degree $n$ approximation as a fibrant replacement) and showed that these model structures are identical by showing they have the same cofibrations and weak equivalences. This model structure is tested.

\begin{xxthm}[\cref{thm: n-poly tested}]
The Barnes--Oman $n$-polynomial model structure in Weiss calculus is a tested Bousfield--Friedlander localization.
\end{xxthm}

\subsection*{Acknowledgements}
This work greatly benefited from (and probably would not exists if it were not for) conversations with Julie Bergner, Rihannon Griffiths, Brenda Johnson and Rekha Santhanam. We thank the referee for a number of helpful comments which greatly improved this article. The author would like to thank the Isaac Newton Institute for Mathematical Sciences, Cambridge, for support and hospitality during the programme \emph{Equivariant homotopy theory in context}, where work on this paper was undertaken. This work was supported by EPSRC grant EP/Z000580/1. The author was supported by the Nederlandse Organisatie voor Wetenschappelijk Onderzoek (Dutch Research Council) Vidi grant no VI.Vidi.203.004 and, during the final stages of this project, by the Engineering and Physical Sciences Research Council (EP/Z534705/1).

\section{Bousfield--Friedlander localization as a left Bousfield localization}
Let $\es{M}$ be a right proper model category. A homotopical endofunctor $Q: \es{M} \to \es{M}$ is a \emph{Quillen idempotent monad}\footnote{A Quillen idempotent monad is also sometimes called a \emph{Bousfield endofunctor}.} if $Q$ preserves homotopy pullback squares, and if there exists a natural transformation $\eta: \bb{1} \to Q$ such that both induced natural transformations $\eta_Q, Q(\eta): Q \to Q^2$ are levelwise weak equivalences. Given this data, there exists a model structure on $\es{M}$ in which the weak equivalences are those $X \to Y$ such that $QX \to QY$ is a weak equivalence in $\es{M}$, and the fibrations are those fibrations $X \to Y$ of $\es{M}$ such that the canonical square
\[\begin{tikzcd}
	X & QX \\
	Y & QY
	\arrow["\eta", from=1-1, to=1-2]
	\arrow[from=1-1, to=2-1]
	\arrow[from=1-2, to=2-2]
	\arrow["\eta"', from=2-1, to=2-2]
\end{tikzcd}\]
is a homotopy pullback square. We denote this new model structure by $\es{M}_Q$ and refer to it as the \emph{Bousfield--Friedlander localization of} $\es{M}$. The original version of this result was proven by Bousfield and Friedlander~\cite[Theorem A.7]{BousfieldFriedlander} and generalised by Bousfield~\cite[Theorem 9.7]{Bousfield}. The version we have stated above is due to Bandklayder, Bergner, Griffiths, Johnson and Santhanam~\cite[Theorem 1.1]{BBGJS2} and crucially uses right properness.

For the rest of this note, we will let $\es{C}$ be an essentially small simplicial category, $\es{D}$ a simplicial right proper model category and will denote by $\Fun(\es{C}, \es{D})$ the category of simplicial functors from $\es{C}$ to $\es{D}$. Without further comment we will equip $\Fun(\es{C}, \es{D})$ with the projective model structure, which exists by~\cite[Proposition A.3.3.2]{HTT}, see also~\cite{Moser}. With these assumptions we can introduce the notion of a \emph{tested} Quillen idempotent monad or equivalently a tested Bousfield--Friedlander localization, as originally introduced in the work of Bandklayder, Bergner, Griffiths, Johnson and Santhanam~\cite[Definition 8.1]{BBGJS2}.  

Since $\es{D}$ is a simplicial model category, it is cotensored over the category $\sS$ of simplicial sets. However, the functor category $\Fun(\es{C},\es{D})$ need not be cotensored over the functor category $\Fun(\es{C}, \sS)$. This error can be corrected via the evaluated cotensor: for $F: \es{C} \to \es{D}$ and $X: \es{C} \to \sS$, we define the \emph{evaluated cotensor} $F^X$ as the equaliser
\begin{equation}\label{def: eval cotensor}
 F^X = \int_{C \in \es{C}} FC^{XC} = \sf{eq} \left( \begin{tikzcd}
	{\underset{C \in \es{C}}{\prod} FC^{XC}} & {\underset{C,C' \in \es{C}}{\prod} (FC'^{XC})^{\es{C}(C,C')}}
	\arrow[shift left, from=1-1, to=1-2]
	\arrow[shift right, from=1-1, to=1-2]
\end{tikzcd}\right).   
\end{equation}

A comprehensive account of the properties and usefulness of the evaluated cotensor construction are given in~\cite{BBGJS1,BBGJS2}.

\begin{definition}\label{def: tested}
A Quillen idempotent monad $Q$ on $\Fun(\es{C},\es{D})$ is said to be \emph{tested} if there exists a set $T(Q)$ of morphisms in $\Fun(\es{C}, \sS)$ such that, for each fibration $F \rightarrow G$ in $\Fun(\es{C},\es{D})$, the diagram
\[\begin{tikzcd}
	F & QF \\
	G & QG
	\arrow[from=1-1, to=1-2]
	\arrow[from=1-1, to=2-1]
	\arrow[from=1-2, to=2-2]
	\arrow[from=2-1, to=2-2]
\end{tikzcd}\]
is a homotopy pullback in $\Fun(\es{C},\es{D})$ if and only if, for every $X \to Y$ in $T(Q)$ the diagram
\[\begin{tikzcd}
	{F^Y} & {F^X} \\
	{G^Y} & {G^X}
	\arrow[from=1-1, to=1-2]
	\arrow[from=1-1, to=2-1]
	\arrow[from=1-2, to=2-2]
	\arrow[from=2-1, to=2-2]
\end{tikzcd}\]
is a homotopy pullback in $\es{D}$. We will refer to the set $T(Q)$ as a set of \emph{test morphisms} and a Bousfield--Friedlander localization with a set of test morphisms will be called \emph{tested}.
\end{definition}

We now show that a tested Bousfield--Friedlander localization may be equivalently described as a left Bousfield localization ``at'' the set of test morphisms, where ``at'' has to be suitably interpreted since the test morphisms are not morphisms in the correct category. 

The proof makes use of what we call \emph{cofibrant homotopy generators}. These were introduced by Dugger~\cite{Dugger} and are a set $\es{G}$ of cofibrant objects that detect weak equivalences in the sense that $X \to Y$ is a weak equivalence in $\es{M}$ if and only if 
\[
\Map_\es{M}(A,X) \longrightarrow \Map_\es{M}(A, Y)
\]
is a weak equivalence of simplicial sets. For instance the set $\{S^n \mid n \geq 0\}$ of simplicial spheres is a set of cofibrant homotopy generators for the Kan-Quillen model structure on simplicial sets. We state and prove the following result for the projective model structure on $\Fun(\es{C}, \es{D})$ but the proof makes it clear that it holds in any model structure on $\Fun(\es{C}, \es{D})$ which is proper and cellular.

\begin{thm}\label{thm: BF and LBL}
Let $Q$ be a tested Quillen idempotent monad on $\Fun(\es{C}, \es{D})$. If $\es{D}$ is equipped with a proper cellular simplicial model structure, then there exists a set $T_\es{D}(Q)$ of morphisms in $\Fun(\es{C}, \es{D})$ such that the Bousfield--Friedlander localization $\Fun(\es{C}, \es{D})_Q$ is the left Bousfield localization of the projective model structure on $\Fun(\es{C}, \es{D})$ at the set $T_\es{D}(Q)$.
\end{thm}
\begin{proof}
Define
\[
T_\es{D}(Q)  =\{ G \otimes X \to G \otimes Y \mid G \in \es{G}, X \to Y \in T(Q)\},
\]
where $\es{G}$ is a set of cofibrant homotopy generators for $\es{D}$ in the sense of~\cite[Proposition A.5]{Dugger}, see also~\cite[Lemma 3.4.5]{Balchin}, and $G \otimes X$ is given by $(G \otimes X)(C) = G \otimes X(C)$, where $\otimes$ denotes the tensoring of $\es{D}$ over simplicial sets. The existence of the set $T_\es{D}(Q)$ follows from the assumption that $\es{D}$ is left proper and cofibrantly generated. The assumptions on $\es{D}$ guarantee that the projective model structure on $\Fun(\es{C}, \es{D})$ is left proper and cellular by combining~\cite[Proposition 4.1.5]{Hirschhorn} with~\cite[Theorem 11.7.3]{Hirschhorn}. As such,~\cite[Theorem 4.1.1]{Hirschhorn} yields that the left Bousfield localization of $\Fun(\es{C}, \es{D})$ at the set $T_\es{D}(Q)$ exists.

It is left to show that the left Bousfield localized model structure agrees with the Bousfield--Friedlander model structure. Both model structures have the same cofibrations, namely those of the underlying model structure on $\Fun(\es{C},\es{D})$, so it suffices by~\cite[Proposition 2.1.11]{Balchin} to show that they have the same fibrant objects. 

The fibrant objects in the left Bousfield localization are the $T_\es{D}(Q)$-local functors: a functor $F:\es{C} \to \es{D}$ is $T_\es{D}(Q)$-local if and only if $F$ is underlying fibrant and the induced map
\[
\Map_{\Fun(\es{C}, \es{D})}(G \otimes Y, F) \longrightarrow \Map_{\Fun(\es{C}, \es{D})}(G \otimes X, F) 
\]
for $G \otimes X \to G \otimes Y$ in $T_\es{D}(Q)$ is a weak equivalence of simplicial sets by~\cite[Example 17.1.4]{Hirschhorn}.
By adjunction~\cite[Proposition 2.13]{BBGJS2}, this is equivalent to asking that the map
\[
\Map_{\es{D}}(G, F^Y) \longrightarrow \Map_{\es{D}}(G, F^X),
\]
be a weak equivalence of simplicial sets. Since $\es{G}$ is a set of cofibrant homotopy generators, it follows from~\cite[Proposition A.5]{Dugger} that the above map is a weak equivalence if and only if $F^Y \to F^X$ is a weak equivalence. Hence a functor $F$ is $T_\es{D}(Q)$-local if and only if $F$ is underlying fibrant and $F^Y \to F^X$ is a weak equivalence for every $X \to Y$ in $T(Q)$.

On the other hand, $F$ is fibrant in the Bousfield--Friedlander localization if and only if $F$ is underlying fibrant and the square
\[\begin{tikzcd}
	F & QF \\
	\ast & {Q(\ast)}
	\arrow[from=1-1, to=1-2]
	\arrow[from=1-1, to=2-1]
	\arrow[from=1-2, to=2-2]
	\arrow[from=2-1, to=2-2]
\end{tikzcd}\]
is a homotopy pullback square. By the existence of a set $T(Q)$ of test morphisms, this is equivalent to the square 
\[\begin{tikzcd}
	{F^Y} & {F^X} \\
	{\ast^Y} & {\ast^X}
	\arrow[from=1-1, to=1-2]
	\arrow[from=1-1, to=2-1]
	\arrow[from=1-2, to=2-2]
	\arrow[from=2-1, to=2-2]
\end{tikzcd}\]
being a homotopy pullback square for all test morphisms $X \to Y$. 
Since $\ast^X =\ast$ by the definition of the evaluated cotensor given in~\eqref{def: eval cotensor}, the bottom map is the unique map $\ast \to \ast$ and hence by~\cite[Proposition 2.3]{BBGJS2} the top horizontal map is an equivalence.  Hence both notions of fibrant objects agree.
\end{proof}

An application of~\cite[Theorem 4.1.1]{Hirschhorn} yields the following corollary, which strengthens the main theorem of~\cite{BBGJS2}. The fact that the Bousfield--Friedlander localization is a left Bousfield localization produces a left proper cellular simplicial model category, whereas the Bousfied--Friedlander localization itself produces a right proper cofibrantly generated simplicial model category, see~\cite[Theorem 1]{BBGJS2}. In general, left Bousfield localizations need not be right proper, see e.g.,~\cite[Remark 7.1 ]{TaggartLocalizations}. To perform subsequent localizations, e.g., in~\cref{thm: homog as RBL}, we thus require both perspectives.

\begin{cor}\label{cor: cof gen}
Let $Q$ be a tested Quillen idempotent monad on $\Fun(\es{C}, \es{D})$. If $\es{D}$ is equipped with a proper cellular simplicial model structure, then the Bousfield--Friedlander localization $\Fun(\es{C}, \es{D})_Q$ is a proper cellular simplicial model category. In particular, the Bousfield--Friedlander localization is cofibrantly generated.
\end{cor}

\begin{rem}\label{rem: hypotheses}
Our hypotheses are a little different to those of Bandklayder, Bergner, Griffiths, Johnson and Santhanam~\cite[Theorem 8.3]{BBGJS2}. Firstly, because of the known existence results for left Bousfield localizations~\cite[Theorem 4.1.1]{Hirschhorn} we have to require that $\es{D}$ is left proper and cellular. Our right properness assumption is only used in an application of ~\cite[Proposition 2.3]{BBGJS2}. On the other hand Theorem $8.3$ of \emph{loc. cit.} has many assumptions (see Conventions $2.8$, $3.1$, $5.7$, and $6.2$ of \emph{loc. cit.}) but upon close examination of the proof the authors only use that $\es{D}$ is a cofibrantly generated right proper simplicial model category such that the set $J_\es{D}$ of generating cofibrations of $\es{D}$ may be chosen such that the codomain of each map is finite relative to $J_\es{D}$, and that $\Fun(\es{C}, \es{D})$ is cofibrantly generated right proper with fibrations also projective fibrations. This overstatement of the conventions required for Theorem $8.3$ of \emph{loc. cit.} follows from the intended applications, i.e., these assumptions are required when considering the homotopy functor and degree $n$ model structures. Note also that our proof does not require the fibrations of $\Fun(\es{C}, \es{D})$ to be projective fibrations. This assumption in \emph{loc. cit.} is an artifact of the method of proof via explicitly constructing a set of generating acyclic cofibrations.
\end{rem}

\section{An application to functor calculus}
Recall from the introduction that to a functor $F$, a functor calculus assigns a tower of functors approximating $F$ in a way analogous to the Taylor series for a function. The functor $P_nF$ is the \emph{universal degree $n$} functor under $F$, where the definition of degree $n$ is dependent on the choice of calculus. We will always assume that $P_nF$ is fibrant, since this can be arranged~\cite{BBGJS2}. Throughout this section we will assume the existence of such a calculus and a corresponding degree $n$ model structure which comes to us as a tested Bousfield--Friedlander localization. In all examples this Bousfield--Friedlander localization is with respect to the universal degree $n$ approximation $P_n$.

The Bousfield--Friedlander localization of a suitable model structure on a category of functors yields a model structure in which the fibrant objects are precisely the degree $n$ functors. This approach has seen success in Goodwillie calculus by  Biedermann, Chorny, and R\"ondigs~\cite{BiedermannChornyRondigs, BiedermannRondigs} and in Weiss calculus by Barnes and Oman~\cite{BarnesOman}. The seminal work of Bandklayder, Bergner, Griffiths, Johnson and Santhanam~\cite{BBGJS2} shows that the Bousfield--Friedlander localization in Goodwillie calculus is tested, and extends the theory to construct a degree $n$ model structure for the discrete calculus of Bauer, Johnson, and McCarthy~\cite{BauerJohnsonMcCarthy} as a tested Bousfield--Friedlander localization. In~\cref{sec: Weiss calc} we show that the corresponding localization in Weiss calculus is also tested. We summarise these results here.

\begin{lem}\label{lem: existence of degree n tested}
Let $\es{C}$ be an essentially small simplicial category and let $\es{D}$ be a proper cellular simplicial model category.
\begin{enumerate}
    \item Let $\es{C}$ be a full simplicial subcategory of a cofibrantly generated model category on cofibrant and simplicially finitely presentable objects which is closed under finite homotopy colimits and has a terminal object. Moreover, assume that the functorial fibrant replacement of an object in $\es{C}$ is a sequential colimit of objects in $\es{C}$. Let $\es{D}$ be such that weak equivalences, fibrations and homotopy pullbacks are preserved under sequential colimits. Then the $n$-excisive model structure for Goodwillie calculus is a tested Bousfield--Friedlander localization. In particular, the $n$-excisive model structure is a left Bousfield localization. 
    \item If $\es{C}$ has finite coproducts and a terminal object and $\es{D}$ is a stable model category, then the degree $n$ model structure for Goodwillie calculus is a tested Bousfield--Friedlander localization. In particular, the degree $n$ model structure is a left Bousfield localization.
\end{enumerate}
\end{lem}
\begin{proof}
The model structures are tested Bousfield--Friedlander localizations by~\cite[Proposition 9.2]{BBGJS2} and~\cite[Lemma 10.3]{BBGJS2}. The left Bousfield localization description follows from~\cref{cor: cof gen} since $\es{D}$ is additionally left proper and cellular.
\end{proof}

\begin{rem}
The assumption of the first part of~\cref{lem: existence of degree n tested} are those of~\cite[Proposition 9.2]{BBGJS2} slightly altered to take into account the fact that we want to identify everything as a left Bousfield localization. For the actual statement of \emph{loc. cit.} one requires Conventions 6.2, 5.7, 3.1, and 2.8 of~\cite{BBGJS2}. Some of these we have been able to cut down thank to assuming that $\es{D}$ is a proper cellular simplicial model category, e.g., ~\cite[Convention 6.2]{BBGJS2} asserts that the set $J$ of generating cofibrations of $\es{D}$ can be chosen such that the codomain of each map is finite relative to $J$. This is immediate from $\es{D}$ being cellular, see e.g.,~\cite[Theorem 12.4.4]{Hirschhorn}. It is certainly possible to imagine a world where some of the hypotheses on $\es{C}$ are unnecessary but we do not pursue this here.
\end{rem}

A functor $F: \es{C} \to \es{D}$ is $n$-homogeneous if it is degree $n$ (i.e., $F \to P_nF$ is an equivalence) and \emph{$n$-reduced} (i.e., the degree $(n-1)$ approximation $P_{n-1}F$ is trivial). By utilising the description of a tested Bousfield--Friedlander localization as a left Bousfield localization in~\cref{thm: BF and LBL} we can produce a model structure in which the bifibrant objects are the $n$-homogeneous functors.

\begin{thm}\label{thm: homog as RBL}
Let $\es{D}$ be a proper cellular simplicial model category. If the degree $n$ and degree $(n-1)$ model structures exist as tested Bousfield--Friedlander localizations, then there is a model structure on $\Fun(\es{C}, \es{D})$ in which the cofibrant objects are the underlying cofibrant $n$-reduced functors and the fibrations are the fibrations of the degree $n$ model structure. In particular the bifibrant objects are the bifibrant $n$-homogeneous functors. We call this the $n$-homogeneous model structure.
\end{thm}
\begin{proof}
Let $T(P_n)$ denote the set of test morphisms for the degree $n$ model structure. By~\cref{thm: BF and LBL}, the degree $n$ model structure is the left Bousfield localization at the set 
\[
T_\es{D}(P_n) = \{G \otimes T  \mid G \in \es{G}, T \in T(P_n) \},
\]
where $\es{G}$ is a set of cofibrant homotopy generators for $\es{D}$. It follows from~\cref{cor: cof gen} that the degree $n$ model structure is a right proper cellular model category. Let 
\[
K(P_{n-1}) = \{\hocof(T) \mid T \in T(P_{n-1})\}
\]
be the set of homotopy cofibres of the set of test morphisms for the degree $(n-1)$ model structure, and let 
\[
K_\es{D}(P_{n-1}) = \{G \otimes K \mid  G \in \es{G}, K \in K(P_n)\}.
\]
We claim that the required model structure is the right Bousfield localization of the degree $n$-model structure at the set $K_\es{D}(P_{n-1})$. The right Bousfield localization exists by~\cite[Theorem 5.1.1]{Hirschhorn} and the fibrations follow immediately from \emph{loc. cit.}, hence it suffices to show that the cofibrant objects are the underlying cofibrant $n$-reduced functors. 

We will show that a map $E \to F$ is a cofibration in the $n$-homogeneous model structure if and only if $E \to F$ is a underlying cofibration and a weak equivalence in the degree $(n-1)$ model structure. By definition a map $E \to F$ is an n-homogeneous cofibration if and only if it has the left lifting property with respect to $n$-homogeneous acyclic fibrations and a lifting argument analogous to~\cite[Lemma 8.5]{TaggartUnitary}, implies that it suffices to show that a map $X \to Y$  between degree $n$ functors is an acyclic fibration in the $n$-homogeneous model structure if and only if $X \to Y$ is a fibration between degree $n$ functors in the degree $(n-1)$ model structure.

Arguing analogously to~\cite[Proposition 8.3]{TaggartUnitary} we see that a map $X \to Y$ is an acyclic fibration in the $n$-homogeneous model structure if and only if $X \to Y$ is a fibration in the degree $(n-1)$ model structure and a weak equivalence in the $n$-homogeneous model structure. So to complete the proof it suffices to show that if $X \to Y$ is a fibration in the degree $(n-1)$ model structure between degree $n$ functors, then $X \to Y$ is a weak equivalence in the $n$-homogeneous model structure.

Consider the commutative diagram
\[\begin{tikzcd}
	X & {P_nX} & {P_{n-1}X} \\
	Y & {P_nY} & {P_{n-1}Y}
	\arrow["\simeq", from=1-1, to=1-2]
	\arrow[from=1-1, to=2-1]
	\arrow[from=1-2, to=1-3]
	\arrow[from=1-2, to=2-2]
	\arrow[from=1-3, to=2-3]
	\arrow["\simeq", from=2-1, to=2-2]
	\arrow[from=2-2, to=2-3]
\end{tikzcd}\]
in which the left-most horizontal maps are equivalences since $X$ and $Y$ are assumed degree $n$, and hence the left square is a homotopy pullback square. Since $X \to Y$ is assumed to be a fibration in the degree $(n-1)$-model structure the outer-most square is a homotopy pullback diagram, thus the right-most square is also a homotopy pullback by definition of the Bousfield--Friedlander localized model structure. Since $\Fun(\es{C}, \es{D})$ is right proper, it now suffices to show that the map $P_{n-1}X \to P_{n-1}Y$ is a weak equivalence in the $n$-homogeneous model structure.

By the general theory of right Bousfield localizations, a map $X \to Y$ is a weak equivalence in the $n$-homogeneous model structure if and only if the induced map
\[
\Map_{\Fun(\es{C}, \es{D})}(G \otimes K, P_nX) \longrightarrow \Map_{\Fun(\es{C}, \es{D})}(G \otimes K, P_nY),
\]
is a weak equivalence of simplicial sets. The map $P_{n-1}X \to P_{n-1}Y$ is a weak equivalence in the $n$-homogeneous model structure since degree $(n-1)$ functors are trivial in the $n$-homogeneous model structure. Indeed, for every $T: A \to B$ in $T(P_{n-1})$, there is a homotopy fibre sequence
\[
\Map_{\Fun(\es{C}, \es{D})}(G \otimes \hocof(T), F) \longrightarrow \Map_{\Fun(\es{C}, \es{D})}(G \otimes B, F) \longrightarrow \Map_{\Fun(\es{C}, \es{D})}(G \otimes A, F),
\]
in which the right-most map is an equivalence if $F$ is a degree $(n-1)$ functor by the fact that $T_\es{D}(P_{n-1})$ is a set of test morphisms for the degree $(n-1)$ model structure. It follows that for every $G \otimes \hocof(T) \in K_\es{D}(P_{n-1})$ the simplicial set $\Map_{\Fun(\es{C}, \es{D})}(G \otimes \hocof(T), F)$ is contractible and hence $F$ is trivial in the $n$-homogeneous model structure.
\end{proof}

We now specialise to two particular calculi, namely the discrete calculus of Bauer, Johnson and McCarthy~\cite{BauerJohnsonMcCarthy}, and Goodwillie calculus~\cite{GoodCalcIII}. These calculi have tested degree $n$-model structures by~\cref{lem: existence of degree n tested}, and we now prove that there is a Quillen adjunction between the $n$-homogeneous model structures which cannot be a Quillen equivalence, i.e., that homotopically, these calculi are genuinely different theories. This partially answers one of the motivating questions of \cite{BBGJS2}.

We denote by $\Gamma_n$ the universal degree $n$ approximation in discrete calculus defined for instance in~\cite[Definition 3.12]{BBGJS2}. For a functor $F$, the functor $\Gamma_nF$ is the universal functor for which the $(n+1)$-st cross-effect vanish, and is a Quillen idempotent monad by~\cite[Theorem 7.2]{BBGJS2}. We denote by $P_n$ the universal $n$-excisive approximation in Goodwillie calculus, defined for instance in~\cite[Definition 3.2]{BBGJS2}. For a functor $F$, the functor $P_nF$ is the universal functor which sends strongly homotopy cocartesian $(n+1)$-cubes to homotopy cartesian diagrams and is a Quillen idempotent monad by~\cite[Theorem 5.8]{BiedermannRondigs}.

\begin{lem}\label{cor: homog discerete Goodwillie}
Under the hypotheses of~\cref{lem: existence of degree n tested}, there is a Quillen adjunction
\[
\adjunction{\id}{R_{K_{\es{D}}(\Gamma_{n-1})}(\Fun(\es{C}, \es{D})_{\Gamma_n})}{R_{K_{\es{D}}(P_{n-1})}(\Fun(\es{C}, \es{D})_{P_n})}{\id}
\]
between the $n$-homogeneous model structure for discrete calculus and the $n$-homogeneous model structure for Goodwillie calculus.
\end{lem}
\begin{proof}
By \cite[Proposition 3.3.18(1) and Theorem 3.1.6(1)]{Hirschhorn}, there is a Quillen adjunction
\[
\adjunction{\id}{\Fun(\es{C}, \es{D})_{\Gamma_n}}{\Fun(\es{C}, \es{D})_{P_n}}{\id}
\]
between the degree $n$ model structures, since
\[
\adjunction{\id}{\Fun(\es{C}, \es{D})}{\Fun(\es{C}, \es{D})_{P_n}}{\id},
\]
is a Quillen adjunction by~\cref{thm: BF and LBL}, and the right Quillen functor $\id: \Fun(\es{C}, \es{D})_{P_n} \to \Fun(\es{C}, \es{D})$ sends every $n$-excisive functor to a degree $n$ functor in discrete calculus by \cite[Proposition 3.3]{GoodCalcIII}.

By the hypothesis of~\cref{lem: existence of degree n tested}, the model category $\es{D}$ is stable, and the sets of $T_{\es{D}}(\Gamma_n)$-local and $T_\es{D}(P_n)$-local objects are closed under suspension, so the degree $n$ model structures in both calculi are stable model categories by \cite[Proposition 3.6]{BarnesRoitzheim}. We will thus construct the desired Quillen adjunction as an application of the Cellularization Principle \cite[Theorem 2.1(1)]{GreenleesShipley}.

By~\cite[Definition 10.2, Lemma 10.3]{BBGJS2}, the test morphisms for degree $(n-1)$ model structure in discrete calculus are given by
\[
T(\Gamma_{n-1}) = \left\{ \underset{S \subseteq \underline{n}}{\hocolim}~ R^{\bigvee_{i \in \underline{n}\setminus S} A}   \longrightarrow  R^{\bigvee_{i =1}^{n} A}  \mid A \in \mathsf{sk}(\es{C}) \right\},
\]
where $R^X$ is the representable functor on $X$ and $\mathsf{sk}(\es{C})$ denotes the skeleton of $\es{C}$. Note that $T(\Gamma_{n-1})$ is a set since $\es{C}$ is essentially small, and hence $\sf{sk}(\es{C})$ is a set of representatives of isomorphism classes. It follows from a routine calculation, see e.g.,~\cite[Example 5.9.4]{MunsonVolic} that
\[
K_\es{D}(\Gamma_{n-1}) = \left\{ G \otimes \bigwedge_{i=1}^n R^A \mid A \in \sf{sk}(\es{C}), \ G \in \es{G}\right\},
\]
where $\es{G}$ is a set of homotopy generators for $\es{D}$. The right Bousfield localization of the degree $n$ model structure for discrete calculus at the set $K_\es{D}(\Gamma_{n-1})$ yields the $n$-homogeneous model structure for discrete calculus by~\cref{thm: homog as RBL}. By \cite[Theorem 6.4]{BiedermannRondigs}, the right Bousfield localization of the degree $n$ model structure for Goodwillie calculus at the set $K_\es{D}(\Gamma_{n-1})$ also yields the $n$-homogeneous model structure for Goodwillie calculus.

Denote by $QK_\es{D}(\Gamma_{n-1})$ the set of cofibrant replacements of the objects of $K_\es{D}(\Gamma_{n-1})$ in the projective model structure on $\Fun(\es{C}, \es{D})$. The cellularization principle, \cite[Theorem 2.1(1)]{GreenleesShipley} provides a Quillen adjunction
\[
\adjunction{\id}{R_{QK_{\es{D}}(\Gamma_{n-1})}(\Fun(\es{C}, \es{D})_{\Gamma_n})}{R_{QK_{\es{D}}(P_{n-1})}(\Fun(\es{C}, \es{D})_{P_n})}{\id}
\]
between the right Bousfield localizations at $QK_{\es{D}}(\Gamma_{n-1})$ of the degree $n$ and $n$-excisive model structures. As such, to complete the proof, it suffices to show that there is an equality of model structures between the right Bousfield localization of a model category $\es{M}$ at a set $K$ and the right Bousfield localization of $\es{M}$ at the set $QK$. This equality follows since the fibrations of both model structures agree with the fibrations of $\es{M}$, and the weak equivalences agree since these are tested by mapping out of the objects of $K$ and $QK$ respectively and these agree on (right) homotopy function complexes, see e.g.,~\cite[Definition 17.2.1]{Hirschhorn}, since projectively cofibrant objects are cofibrant in the respective degree $n$ model structures.
\end{proof}

\begin{rem}\label{rem:different}
The under the hypotheses of~\cref{lem: existence of degree n tested}, the Quillen adjunction
\[
\adjunction{\id}{R_{K_{\es{D}}(\Gamma_{n-1})}(\Fun(\es{C}, \es{D})_{\Gamma_n})}{R_{K_{\es{D}}(P_{n-1})}(\Fun(\es{C}, \es{D})_{P_n})}{\id}
\]
between $n$-homogeneous model structures is not, in general, a Quillen equivalence. It suffices to consider the case $n=1$, and show that the derived unit $F \to P_nF$ need not be an equivalence on a $1$-homogeneous functor in discrete calculus, i.e., to exhibit a $1$-homogeneous functor in discrete calculus which is not $1$-excisive. A functor $F$ is $1$-homogeneous in discrete calculus if and only if $F(0)\simeq 0$, and the second cross-effect of $F$ vanishes. The second cross-effect is given by the total homotopy fibre of the square
\[\begin{tikzcd}[ampersand replacement=\&,cramped]
	{F(C_1 \sqcup C_2)} \& {F(C_2)} \\
	{F(C_1)} \& {F(0)}
	\arrow[from=1-1, to=1-2]
	\arrow[from=1-1, to=2-1]
	\arrow[from=1-2, to=2-2]
	\arrow[from=2-1, to=2-2]
\end{tikzcd}\]
and so vanishes if and only if the canonical map $F(C_1 \sqcup C_2) \to F(C_1) \times F(C_2)$ is an equivalence, i.e., if and only if $F$ is homotopically additive. It follows that the derived unit is a weak equivalence if and only if every reduced homotopically additive functor is excisive. This is demonstrably false: for example the functor $H_0(-)[0] :\sf{Ch}(\bb{Z}) \to \sf{Ch}(\bb{Z})$ on ($\bb{Z}$-graded) chain complexes of abelian groups, which sends a chain complex to its $0$-th homology group viewed as chain complex in degree $0$ is reduced and additive, but does not preserve homotopy cofibre sequences. Denote by $S^n$ the chain complex $\bb{Z}[n]$ consisting of a single copy of $\bb{Z}$ in degree $n$, and by $D^n$ the chain complex consisting of a copy of $\bb{Z}$ in degree $n$ and degree $n-1$ with differential $d_n$ given by the identity map. Under this functor, the homotopy cofibre sequence
\[
S^0 \hookrightarrow D^1 \longrightarrow S^1
\]
induced by the generating cofibrations of the projective model structure on chain complexes (see e.g,.~\cite[Definition 2.3.3, Theorem 2.3.11]{Hovey}) is sent to the sequence of abelian groups 
\[
\bb{Z}[0] \longrightarrow 0 \longrightarrow 0
\]
which is not a homotopy fibre sequence since the mapping cone of $\bb{Z}[1]$ is not quasi-isomorphic to $0$.
\end{rem}

\begin{rem}
The fact that both homogeneous model structures are right Bousfield localizations (of different model structures) at the same set of objects is reminiscent of the fact that the cubes used to define the colocalising objects generate all strongly homotopy cocartesian cubes under cobase change, see e.g.,\cite[Example 3.2.7 and Lemma 3.2.8]{ABFJ}, and so see enough of the required cubical homotopy theory to be meaningful in both calculi.
\end{rem}

\section{An application to Weiss calculus}\label{sec: Weiss calc}

We can now show that the polynomial model structures for Weiss calculus fit within the framework of~\cite{BBGJS2}. We will only consider Weiss calculus~\cite{WeissOrthog, WeissErratum} for functors with values in the category $\sS$ of simplicial sets since we have worked simplicially throughout this note. The proofs readily extend to the topological setting and to any known variant of Weiss calculus~\cite{TaggartUnitary, TaggartReal, TaggartLocalizations, CarrTaggart}. Extensions to other target categories form part of the work in progress of the author with Bergner, Griffiths, Johnson and Santhanam.

Denote by $\vect{\bb{k}}$ the simplicial category of finite-dimensional inner product spaces over $\bb{k}$ with linear isometric isomorphisms for $\bb{k}$ either the reals, complex numbers or quaternions. The morphism simplicial set $\vect{\bb{k}}(V,W)$ is given by applying the singular simplicial set functor to the Stiefel manifold of $\dim(V)$-frames in $W$.

\begin{definition}
A simplicial functor $F: \vect{\bb{k}} \to \sS$ is \emph{$n$-polynomial} if the canonical map
\[
F(V) \longrightarrow \underset{0 \neq U \subseteq \bb{k}^{n+1}}{\holim}~F(V \oplus U) =: T_nF(V),
\]
is an equivalence, for all $V \in \vect{\bb{k}}$. The homotopy limit is taken over the category of non-zero linear subspaces of $\bb{k}^{n+1}$, see e.g.,~\cite{WeissErratum} for more on this homotopy limit and note that we have used different notation from \emph{loc. cit.} 
\end{definition}

By the (evaluated cotensor version of) Yoneda Lemma~\cite[Lemma 2.14]{BBGJS2} and~\cite[Proposition 2.18]{BBGJS2}, a functor is $n$-polynomial if and only if the map 
\[
F^{\vect{\bb{k}}(V, -)} \longrightarrow F^{\hocolim_{0 \neq U \subseteq \bb{k}^{n+1}}\vect{\bb{k}}(V \oplus U, -)}
\]
is an equivalence for each $V \in \vect{\bb{k}}$. Define the universal $n$-polynomial approximation as 
\[
P_nF = \hocolim( F \to T_nF \to T_n^2 F \to \cdots ).
\]
A proof that $P_nF$ is $n$-polynomial may be found in~\cite{WeissErratum}. Barnes and Oman~\cite[Proposition 6.5]{BarnesOman} constructed an $n$-polynomial model structure for Weiss calculus by Bousfield--Friedlander localizing the projective model structure at the endofunctor $P_n$. They then go on to show~\cite[Proposition 6.6]{BarnesOman} that this model structure is a left Bousfield localization, which corrects the lack of cofibrant generation. Using~\cref{main thm: LBL} we give an alternative proof.

\begin{thm}\label{thm: n-poly tested}
The $n$-polynomial model structure is a tested Bousfield--Friedlander localization with set of test morphisms given by 
\[
T(P_n) =\left\{\underset{0\neq U \subseteq \bb{k}^{n+1}}{\hocolim}~\vect{\bb{k}}(V \oplus U,-) \longrightarrow \vect{\bb{k}}(V,-) \mid V \in \vect{\bb{k}}\right\}.
\]
In particular, the $n$-polynomial model structure is proper and cellular.
\end{thm}
\begin{proof}
With this choice of test morphisms, the discussion directly preceding the theorem statement reduces the proof to showing that for $F \to G$ a projective fibration, the square
\[\begin{tikzcd}
	F & {P_nF} \\
	G & {P_nG}
	\arrow[from=1-1, to=1-2]
	\arrow[from=1-1, to=2-1]
	\arrow[from=1-2, to=2-2]
	\arrow[from=2-1, to=2-2]
\end{tikzcd}\]
is a homotopy pullback if and only if the square 
\[\begin{tikzcd}
	F & {T_nF} \\
	G & {T_nG}
	\arrow[from=1-1, to=1-2]
	\arrow[from=1-1, to=2-1]
	\arrow[from=1-2, to=2-2]
	\arrow[from=2-1, to=2-2]
\end{tikzcd}\]
is a homotopy pullback. The proof of this biconditional statement follows analogously to~\cite[Proposition 9.2]{BBGJS2} since \emph{loc. cit.} only uses intrinsic properties of a functor calculus.
\end{proof}

\bibliography{references}
\bibliographystyle{alpha}
\end{document}